\documentclass[a4paper,11pt]{amsart}
\usepackage[left=2.7cm,right=2.7cm,top=3.5cm,bottom=3cm]{geometry}

\usepackage{amsthm,amssymb,amsmath,amsfonts,mathrsfs,amscd,amsbsy,dsfont,verbatim}
\usepackage{stmaryrd,dsfont}
\usepackage[new]{old-arrows}
\usepackage[latin1]{inputenc}
\usepackage[all,cmtip]{xy}
\usepackage{latexsym}
\usepackage{longtable}
\usepackage{mathtools}
\usepackage{marginnote}
\usepackage{graphicx}
\usepackage{dutchcal}

\newcounter{lettera}

\usepackage[pagebackref]{hyperref}

\mathtoolsset{showonlyrefs}

\usepackage{graphicx}
\newcommand{\Bmu}{\mbox{$\raisebox{-0.59ex}
  {$l$}\hspace{-0.18em}\mu\hspace{-0.88em}\raisebox{-0.98ex}{\scalebox{2}
  {$\color{white}.$}}\hspace{-0.416em}\raisebox{+0.88ex}
  {$\color{white}.$}\hspace{0.46em}$}{}}

\numberwithin{equation}{section}

\newfont{\cyr}{wncyr10 scaled 1100}
\newfont{\cyrr}{wncyr9 scaled 1000}

\theoremstyle{plain}
\newtheorem{theorem}{Theorem}[section]

\newtheorem{proposition}[theorem]{Proposition}
\newtheorem{lemma}[theorem]{Lemma}
\newtheorem{corollary}[theorem]{Corollary}

\theoremstyle{definition}
\newtheorem{definition}[theorem]{Definition}
\newtheorem{assumption}[theorem]{Assumption}

\theoremstyle{remark}
\newtheorem{remark}[theorem]{Remark}

\newtheorem{remark/notation}[theorem]{Remark/Notation}
\newtheorem{notation/convention}[theorem]{Notation/Convention}

\newcommand{\Q}{\mathds Q}
\newcommand{\N}{\mathds N}
\newcommand{\Z}{\mathds Z}

\newcommand{\R}{\mathds R}
\newcommand{\C}{\mathds C}

\newcommand{\F}{\mathds F}
\newcommand{\T}{\mathds T}

\newcommand{\defeq}{\vcentcolon=}

\DeclareMathOperator{\Frob}{Frob}

\DeclareMathOperator{\Hom}{Hom}

\DeclareMathOperator{\Gal}{Gal}
\DeclareMathOperator{\GL}{GL}

\DeclareMathOperator{\SL}{SL}
\DeclareMathOperator{\Sel}{Sel}

\DeclareMathOperator{\CH}{CH}
\DeclareMathOperator{\AJ}{AJ}

\DeclareMathOperator{\im}{im}

\DeclareMathOperator{\f}{\boldsymbol f}

\newcommand{\cores}{\mathrm{cores}}

\newcommand{\tr}{\mathrm{tr}}
\newcommand{\ord}{\mathrm{ord}}

\newcommand{\an}{\mathrm{an}}

\newcommand{\cg}{\mathcal{g}}
\newcommand{\cCH}{\mathcal{CH}}

\newcommand{\Sha}{\mbox{\cyr{X}}}

\usepackage[usenames]{color}
\definecolor{Indigo}{rgb}{0.2,0.1,0.7}
\definecolor{Violet}{rgb}{0.5,0.1,0.7}
\definecolor{White}{rgb}{1,1,1}
\definecolor{Green}{rgb}{0.1,0.9,0.2}

\newcommand{\longepi}{\mbox{\;$\relbar\joinrel\twoheadrightarrow$\;}}

\newfont{\gotip}{eufb10 at 12pt}

\newcommand{\cO}{{\mathcal O}}

\newcommand{\m}{\mathfrak{m}}
\newcommand{\p}{\mathfrak{p}}

\newcommand{\fP}{\mathfrak{P}}

\newcommand{\hf}{\boldsymbol f^{(p)}}
\newcommand{\hnf}{f^{(p)}}

\DeclareMathOperator{\GS}{GS}

\makeatletter
\@namedef{subjclassname@2020}{%
  \textup{2020} Mathematics Subject Classification}
\makeatother

\begin{document}

\title[Analytic rank one propagation in Hida families]{Analytic rank one propagation in Hida families}
\author{Stefano Vigni}

\thanks{The research by the author is partially supported by the GNSAGA group of INdAM and by the MUR Excellence Department Project awarded to Dipartimento di Matematica, Universit\`a di Genova, CUP D33C23001110001.}

\begin{abstract}
Let $f$ be a non-CM newform of weight $k\geq4$, level $N$ and trivial Nebentypus. Let $p\nmid N$ be an odd prime number that is ordinary for $f$ and denote by $\hf$ the $p$-adic Hida family passing through $f$. Assuming very specific instances of two general conjectures in arithmetic algebraic geometry (injectivity of $p$-adic Abel--Jacobi maps, positive definiteness of archimedean height pairings \emph{\`a la} Gillet--Soul\'e), we prove that, for all but finitely many $p$ as above, if the analytic rank of $f$ is $1$, then all but finitely many specializations of $\hf$ of even weight and trivial Nebentypus have analytic rank $1$. This result provides evidence for Greenberg's ``minimality conjecture'' on analytic ranks in families of modular forms.
\end{abstract}

\include{thebibliography}

\address{Dipartimento di Matematica, Universit\`a di Genova, Via Dodecaneso 35, 16146 Genova, Italy}
\email{stefano.vigni@unige.it}



\maketitle

\section{Introduction}

Let $p$ be a prime number. The study of the variation of central derivatives of $L$-functions in $p$-adic families of modular forms is a key theme in arithmetic geometry. A guiding principle in this area is a conjecture of Greenberg predicting that the analytic ranks of cusp forms of even weight and trivial Nebentypus in a $p$-adic Hida family should be as small as allowed by the corresponding functional equations, with at most finitely many exceptions (\cite{Greenberg-CRM}).

To date, evidence for this conjecture in the rank $1$ setting has been anchored to weight $2$. In fact, by starting with a newform $f$ of weight $2$, one can exploit the arithmetic of the abelian variety of $\GL_2$-type that is attached to $f$ by the Eichler--Shimura construction, leveraging (among other ingredients) the Gross--Zagier formula for modular abelian varieties over $\Q$ to translate the analytic rank $1$ property into a condition on Heegner points. In \cite{vigni-hida}, assuming the positive definiteness of certain (archimedean) height pairings \emph{\`a la} Gillet--Soul\'e between Heegner cycles, this weight-$2$ machinery was used to obtain a result of the following type in the direction of Greenberg's conjecture: if a Hida family $\f$ contains a weight-$2$ newform with trivial Nebentypus and analytic rank $1$, then infinitely many even-weight specializations of $\f$ with trivial Nebentypus have analytic rank $1$ (for an analogous result for $p$-adic Coleman families, see \cite{PPV}). However, this leaves open a natural structural question: is analytic rank $1$ a property that can be propagated along a family from any higher-weight classical point, without relying on the specific arithmetic of Mordell--Weil groups in weight $2$?

Under suitable technical assumptions, the present paper answers the question above in the affirmative. More precisely, let $f\in S_k(\Gamma_0(N))$ be a non-CM newform of even weight $k\geq4$, level $N$ and trivial Nebentypus; suppose that the analytic rank $r_\an(f)$ of $f$, \emph{i.e.}, the order of vanishing of $L(f,s)$ at $s=k/2$, is $1$. Fix a prime number $p\nmid2N$ that is ordinary for $f$ and assume that the residual $p$-adic Galois representation attached to $f$ is irreducible: this rules out only finitely many ordinary primes for $f$. We consider the $p$-adic Hida family $\hf$ of tame level $N$ passing through $f$ and ask how the analytic rank behaves at the specializations $\hnf_\kappa$ of $\hf$ of even weight $\kappa\geq2$ and trivial Nebentypus: these specializations correspond to the even integers $\kappa\geq2$ such that $\kappa\equiv k\pmod{p-1}$. Greenberg's ``minimality conjecture'' predicts that the generic analytic rank along $\hf$ should be $1$, which means that one should have $r_\an\bigl(\hnf_\kappa\bigr)=1$ for all but finitely many even integers $\kappa\geq2$ such that $\hnf_\kappa$ has trivial Nebentypus. Theorem~\ref{main-thm}, which is our main result, confirms this prediction, conditionally on two hypotheses of independent interest: injectivity of a $p$-adic Abel--Jacobi map on the line spanned by a distinguished imaginary quadratic Heegner cycle (Assumption~\ref{AJ-ass}) and positive definiteness of S.-W. Zhang's height pairings \emph{\`a la} Gillet--Soul\'e on the relevant Kuga--Sato varieties (Assumption~\ref{GS-ass}). Both assumptions are very special instances of general conjectures that are standard, if largely open, in arithmetic algebraic geometry.

From here on, $\kappa\geq2$ varies over the set of integers with $\kappa\equiv k\pmod{p-1}$ and we work under Assumptions \ref{AJ-ass} and \ref{GS-ass}. As in \cite{vigni-hida}, our strategy builds on a combination of the $p$-adic interpolation of Nekov\'a\v{r}'s Heegner cycles (\cite{Nek}) by Howard's big Heegner points (\cite{Howard-Inv}), due to Castella and Ota (\cite{CasHeeg}, \cite{ota-JNT}), with Zhang's formula of Gross--Zagier type for higher-weight modular forms (\cite{Zhang-heights}), which expresses the central derivative $L'\bigl(\hnf_\kappa/K,\kappa/2\bigr)$ in terms of the Gillet--Soul\'e (archimedean) height of a Heegner-type cycle, where $K$ is an auxiliary imaginary quadratic field such that
\begin{itemize}
\item every prime factor of $Np$ splits in $K$;
\item if $\eta_K$ is the Dirichlet character associated with $K$, then $r_\an(f\otimes\eta_K)=0$.
\end{itemize}
Analytic rank~$1$ is propagated from $f$ along $\hf$ by showing, via a torsionfreeness argument for an extended Selmer group over $K$, that the big Heegner cycle over $K$ attached to $\hf$ has non-torsion specializations at all but finitely many $\kappa\geq4$ as above. It follows that, for all such $\kappa$, the Heegner cycle over $K$ attached to $\hnf_\kappa$ is nontrivial, and then the nonvanishing of this cycle is converted into the nonvanishing of $L'\bigl(\hnf_\kappa/K,\kappa/2\bigr)$ via Zhang's formula. Finally, a basic factorization of $L$-functions and a parity argument coming from the constancy of root numbers of higher weight specializations of a Hida family pin the analytic rank of $\hnf_\kappa$ down to $1$, as desired.

While the analytic part of \cite{vigni-hida} shares with the present paper some of its main ingredients (\emph{e.g.}, Zhang's formula for $L$-derivatives, specialization of big Heegner points), the settings of the two articles present various differences, which we would like to highlight:
\begin{itemize}
\item in \cite{vigni-hida}, the initial analytic rank $1$ input for the propagation arguments is in weight $2$ (in fact, the abelian variety attached to the weight-$2$ specialization of the Hida family, whose arithmetic is systematically exploited, is assumed to be an elliptic curve);
\item in \cite{vigni-hida}, the tame level $N$ of the Hida family is taken to be square-free (this allowed us to easily use results of Fischman on the image of $\Lambda$-adic Galois representations (\cite{Fischman}), which are not needed here); 
\item instead of the congruence $\kappa\equiv2\pmod{p-1}$, which detects specializations of even weight and trivial Nebentypus of the Hida family appearing in \cite{vigni-hida}, technical reasons forced us to impose in \cite{vigni-hida} the stronger condition $\kappa\equiv2\pmod{2(p-1)}$, so the main results of \cite{vigni-hida} concern only infinitely many specializations with trivial Nebentypus.
\end{itemize}
On the other hand, no analogue of the nontriviality of the $K$-rational Heegner cycle $z_{f,K}$ under the $p$-adic Abel--Jacobi map $\AJ_{p,K,k}$ (Assumption \ref{AJ-ass}) needs to be imposed in \cite{vigni-hida}; this is due to the fact that the role of $\AJ_{p,K,k}$ is played in \cite{vigni-hida} by a mod-$p$ Kummer map on a Mordell--Weil group of an elliptic curve, which is well known to be injective. 

As a concluding remark, observe that, under similar assumptions, our strategy of proof of Theorem \ref{main-thm} can be adapted to yield an analogous result on the propagation of analytic rank $0$ in Hida families. However, such a rank $0$ statement can also be obtained (with no need for counterparts of Assumptions \ref{AJ-ass} and \ref{GS-ass}) by combining work of Mazur--Kitagawa (\cite{kitagawa}) and of Kato (\cite{Kato}), as explained in \cite[Theorem 7]{Howard-derivatives}. Therefore, we will not elaborate on this case here (see Remark \ref{0-rem} for more details).

The organization of this paper is as follows. Section~\ref{hida-sec} reviews the relevant background on Hida families, big Galois representations and specialization maps. Section~\ref{analytic-sec} contains the proof of our main theorem: \S\S \ref{f-subsec}--\ref{algebraic-subsec} introduce the newform $f$, the $p$-adic Hida family $\hf$, the auxiliary set of primes $\Sigma_f$ and the imaginary quadratic fields that will play a crucial role in our arguments; \S \ref{zhang-subsec} recalls Zhang's formula of Gross--Zagier type for higher (even) weight cusp forms; \S \ref{big-subsec} establishes nontriviality and nontorsionness of the relevant Heegner cycles; finally, \S \ref{main-subsec} assembles these ingredients into the proof of Theorem~\ref{main-thm}.

\subsection{Notation and conventions} \label{nc-subsec}

In this article, $\overline\Q$ is the algebraic closure of $\Q$ inside $\C$ and we write $\iota_\infty:\overline\Q\hookrightarrow\C$ for the corresponding inclusion. For each prime number $p$, we fix an embedding $i_p:\overline\Q\hookrightarrow\overline{\Q_p}$, which determines a prime ideal $\fP$ of the ring of integers of $\overline\Q$ above $p$. Moreover, we fix an embedding $\iota_{p,\infty}:\overline{\Q_p}\hookrightarrow\C$ such that $\iota_{p,\infty}\circ i_p=\iota_\infty$. Finally, $G_\Q\defeq\Gal(\overline\Q/\Q)$ is the absolute Galois group of $\Q$; an analogous notation will be adopted for other number fields as well.

\section{Review of Hida families} \label{hida-sec}

We recall basic notions of Hida's theory of $p$-adic families of modular forms, where $p$ is a prime number. For details and proofs, see, \emph{e.g.}, \cite{hida86b}, \cite{hida86a}, \cite[Ch. 7]{hida-elementary}.

\subsection{Hida families of modular forms} \label{hida-subsec}

Let $N\geq1$ be an integer and let $p$ be a prime number such that $p\nmid N$. Set $\Gamma\defeq1+p\Z_p\subset\Z_p^\times$. 

\subsubsection{$p$-adic Hida families and their specializations}

A \emph{$p$-adic Hida family} of tame level $N$ consists of
\begin{itemize}
\item a complete local noetherian domain $\mathcal R$ that is finitely generated and flat as a module over the Iwasawa algebra $\cO[\![\Gamma]\!]\simeq\cO[\![T]\!]$, where $\cO$ is a suitable finite extension of $\Z_p$;
\item a (dense) collection of distinguished points
\[ \mathcal X^{\text{arith}}\subset\Hom_{\text{cont}}\bigl(\mathcal R,\overline{\Q_p}\,\bigr) \]
called \emph{arithmetic morphisms};
\item a formal $q$-expansion $\f=\f(q)=\sum_{n\geq1}a_n(\f)q^n\in\mathcal R[\![q]\!]$
\end{itemize}
such that for all $\eta\in\mathcal X^{\text{arith}}$ the power series 
\[ f_\eta=f_\eta(q)\defeq\sum_{n\geq1}\eta\bigl(a_n(\f)\bigr)q^n\in\overline{\Q_p}[\![q]\!] \]
is the $q$-expansion of an ordinary cuspidal eigenform on $\Gamma_1(Np^r)$ for some $r=r_\eta\geq1$. The modular form $f_\eta$ is called the \emph{specialization of $\f$ at $\eta$}. By definition, a continuous homomorphism $\eta:\mathcal R\rightarrow\overline{\Q_p}$ is \emph{arithmetic} if the composition 
\[ \Gamma\longrightarrow\mathcal R^\times\overset\eta\longrightarrow\overline{\Q_p}^\times \]
of $\eta$ with the canonical map $\Gamma\rightarrow\mathcal R^\times$ has the form 
\begin{equation} \label{arithmetic-eq}
\gamma\longmapsto\psi(\gamma)\gamma^{k-2}
\end{equation}
for some integer $k\geq2$ and some finite order character $\psi$ of $\Gamma$. The kernels of arithmetic morphisms are called \emph{arithmetic primes} of $\mathcal R$. The integer $k$ in \eqref{arithmetic-eq} is the \emph{weight} of the arithmetic morphism (or of the associated arithmetic prime), while $\psi$ is its \emph{wild character}. In the rest of the present paper, $\mathtt{ArithSpec}(\mathcal R)$ will be the set of all arithmetic primes of $\mathcal R$. Given $\wp\in\mathtt{ArithSpec}(\mathcal R)$, we will write $f_\wp$ for the specialization of $\f$ at the arithmetic morphism whose kernel is $\wp$. Via the embedding $\iota_{p,\infty}$ from \S \ref{nc-subsec}, we always view $f_\wp$ as a cusp form with complex Fourier coefficients.

\begin{remark}
By a slight abuse of terminology, we will refer to the $q$-expansion $\f$ as a $p$-adic Hida family (of tame level $N$) with coefficients in $\mathcal R$. 
\end{remark}

\subsubsection{Specializations at even integers}

Let $\f$ be a $p$-adic Hida family with coefficients in $\mathcal R$ (\emph{cf.} Remark \ref{hida-rem}). Given an even integer $\kappa\geq2$, the continuous homomorphism
\[
\Gamma\longrightarrow\cO^\times,\quad\gamma\longmapsto\gamma^{\kappa-2},
\]
can be extended (essentially by the ``lying over'' theorem in commutative algebra) to an arithmetic morphism $\eta_\kappa:\mathcal R\rightarrow\overline{\Q_p}$ with trivial wild character. For simplicity, with notation as above, we shall call $f_\kappa\defeq f_{\eta_\kappa}$ the \emph{specialization of $\f$ at $\kappa$}.

\subsubsection{$p$-stabilizations and Hida families}

Now let $f\in S_k(\Gamma_1(M))$ be a $p$-ordinary normalized newform of weight $h\geq2$ and level $M\geq1$. Set
\[ 
f_0\defeq\begin{cases}f&\text{if $p\,|\,M$},\\[3mm]f^\sharp&\text{if $p\nmid M$}, \end{cases} \]
where $f^\sharp$ denotes the $p$-stabilization of $f$ (see, \emph{e.g.}, \cite[\S 2.4]{vigni-hida}). The cusp form $f_0$ can be characterized as the unique (normalized) $p$-ordinary eigenform of weight $h$ and level divisible by $p$ with the property that $a_n(f_0)=a_n(f)$ except for those $n$ divisible by $p$ (\cite[Lemma 3.3]{hida-measure}). A fundamental result of Hida (\cite[Corollary 3.7]{hida86a}, \cite[\S 7.3, Theorem 3]{hida-elementary}) ensures that there exists a unique Hida family $\f\in\mathcal R[\![q]\!]$ of tame level $M$ such that $f_0=f_\wp$ for some arithmetic prime $\wp$ of $\mathcal R$; this is often described by saying that $\f$ \emph{passes through} $f$. 

\begin{remark} \label{hida-rem}
In this context, the expression ``$p$-ordinary'' used above is somewhat imprecise: see, \emph{e.g.}, \cite[Remark 2.12]{vigni-hida} for details.    
\end{remark}

\subsection{Big Galois representations} \label{big-subsec}

Let $\f$ be a $p$-adic Hida family of tame level $N$ with coefficients in $\mathcal R$.

\subsubsection{The big Galois representation $\T_{\f}$} \label{big-subsubsec}

By Hida theory (\cite[Theorem 2.1]{hida86b}), there is a ``big'' representation $\T_{\f}$ of $G_\Q$ that (under standard assumptions on residual representations, \emph{cf.} \S \ref{consequences-subsubsec}) is a free $\mathcal R$-module of rank $2$ and satisfies the following property: for every arithmetic prime $\wp$ of $\mathcal R$, the quotient $\T_{\f,\wp}/\wp\T_{\f,\wp}$ is equivalent (after a finite base change) to the dual $V_{f_\wp}^*$ of the $p$-adic representation $V_{f_\wp}$ of $G_\Q$ attached to $f_\wp$ (see, \emph{e.g.}, \cite[(1.5.5)]{NP}). The representation
\[
\rho_{\f}:G_\Q\longrightarrow\GL(\T_{\f})\simeq\GL_2(\mathcal R) 
\]
is unramified outside $Np$ and satisfies the equality
\[ \tr\bigl(\rho_{\f}(\Frob_\ell)\bigr)=a_\ell(\f) \]
for all prime numbers $\ell\nmid Np$, where $\Frob_\ell$ denotes the conjugacy class in $G_\Q$ of an arithmetic Frobenius at $\ell$.

\subsubsection{The residual representation $\overline{\T}_{\f}$}

Let $\m_{\mathcal R}$ be the maximal ideal of $\mathcal R$. The residual representation attached to $\f$ is $\overline{\T}_{\f}\defeq\T_{\f}/\m_{\mathcal R}\T_{\f}$. Under the assumptions stated in \S \ref{consequences-subsubsec}, this is a $2$-dimensional representation of $G_\Q$ over the residue field of $\mathcal R$ that is equivalent (up to a finite base change) to the residual representation attached to $f_\wp$ for every arithmetic prime $\wp$ of $\mathcal R$.

\subsubsection{Specialization maps in cohomology} \label{spec-cohom-subsubsec}

Let $\T_{\f}^\dagger$ be its critical twist of $\T_{\f}$ (see, \emph{e.g.}, \cite[Definition 2.1.3]{Howard-Inv}). For every even integer $\kappa\geq4$ and number field $L$, let 
\begin{equation} \label{spec-eq}
\rho_{L,\kappa}:H^1\bigl(L,\T_{\f}^\dagger\bigr)\longrightarrow H^1\bigl(L,V^\dagger_{f_\kappa}\bigr)
\end{equation}
be the specialization map over $L$ at weight $\kappa$ (see, \emph{e.g.}, \cite[\S 5.1]{vigni-hida}). If $\wp_\kappa\in\mathtt{ArithSpec}(\mathcal R)$ corresponds to $\kappa$, then there is an isomorphism $\T_{\f,\wp_\kappa}^\dagger\big/\wp_\kappa\T_{\f,\wp_\kappa}^\dagger\simeq V_{f_\kappa}^\dagger$, so the map $\rho_{L,\kappa}$ in \eqref{spec-eq} induces a specialization map
\begin{equation} \label{spec-eq2}
\rho_{L,\kappa}:H^1\bigl(L,\T_{\f,\wp_\kappa}^\dagger\bigr)\longrightarrow H^1\bigl(L,V^\dagger_{f_\kappa}\bigr),
\end{equation}
which (by a slight abuse of notation) will be denoted by the same symbol. From now on, we always (tacitly) fix our specialization maps as explained in \cite[\S 5.1]{vigni-hida}.

\section{Analytic rank one propagation} \label{analytic-sec}

In this section, we prove the main result of this paper, \emph{i.e.}, the nonvanishing of central derivatives of $L$-functions of all but finitely many specializations of even weight and trivial Nebentypus within a suitable Hida family.

\subsection{The newform $f$ and the Hida family $\hf$} \label{f-subsec}

We introduce the newform and the Hida family passing through it in terms of which we shall state our main result.

\subsubsection{Analytic ranks and root numbers} \label{analytic-subsubsec}

Quite generally, let $\cg\in S_h(\Gamma_0(N))$ be a newform of weight $h\geq2$, level $N$ and trivial Nebentypus; it is well known that if $\Gamma(s)$ is the classical $\Gamma$-function and $\Lambda(\cg,s)\defeq(2\pi)^{-s}\cdot N^{s/2}\cdot\Gamma(s)\cdot L(\cg,s)$ is the completed $L$-function of $\cg$, then there is a functional equation
\[
\Lambda(\cg,s)=W(\cg)\cdot\Lambda(\cg,h-s),
\]
where $W(\cg)\in\{\pm1\}$ is the \emph{root number} of $\cg$. 

\begin{definition}
The \emph{analytic rank} of $\cg$ is $r_\an(\cg)\defeq\ord_{s=h/2}L(\cg,s)\in\N$.
\end{definition}

It can be shown that $W(g)$ controls the parity of $r_\an(\cg)$, \emph{i.e.}, $W(\cg)=(-1)^{r_\an(\cg)}$.

\subsubsection{The newform $f$} \label{f-subsubsec}

Let us fix a newform $f\in S_k(\Gamma_0(N))$ of even weight $k\geq2$, level $N\geq3$ and trivial Nebentypus; the $q$-expansion of $f$ will be denoted by $f(q)=\sum_{n\geq1}a_n(f)q^n$. Throughout this paper, we assume that
\begin{itemize}
\item $f$ has no complex multiplication (CM) in the sense of \cite[p. 34, Definition]{ribet};
\item $r_\an(f)=1$.
\end{itemize}
Let $\Q_f\defeq\Q\bigl(a_n(f)\mid n\geq1\bigr)\subset\overline\Q$ be the Hecke field of $f$, which is a totally real number field. Denote by $\cO$ the ring of integers of $\Q_f$, write $\cO_{\mathfrak l}$ for the completion of $\cO$ at a prime $\mathfrak l$ of $\Q_f$ and for a prime number $\ell$ set $\cO_\ell\defeq\cO\otimes_\Z\Z_\ell=\prod_{\mathfrak l\mid\ell}\cO_{\mathfrak l}$. Let us fix once and for all a field embedding $\Q_f\hookrightarrow\R$.

\begin{remark}
A condition, which we will not assume in this paper, that guarantees that $f$ is not a CM form is that $N$ be square-free.
\end{remark}

\subsubsection{The Hida family $\hf$}

Let $f$ be the newform from \S \ref{f-subsubsec}. Let $p$ be a prime number such that 
\begin{itemize}
\item $p\nmid2N$; 
\item $f$ is $p$-ordinary.    
\end{itemize}
By analogy with what happens in weight $2$, we expect the second condition to be satisfied by infinitely many $p$. Let $\hf$ be the $p$-adic Hida family passing through $f$ (or, rather, through the $p$-stabilization of $f$, \emph{cf.} \S \ref{hida-subsec}) and for every even integer $\kappa\geq2$ write $\hnf_\kappa$ for the specialization of $\hf$ at $\kappa$. 

\subsubsection{Big image} \label{big-image-subsubsec}

Let $p$ be a prime number. Denote by
\[ \rho_{f,p}:G_\Q\longrightarrow\GL_2(\cO_p) \] 
the $p$-adic Galois representation attached to $f$ and $p$. We say that $\rho_{f,p}$ has \emph{big image} if there is an inclusion
\[ 
\bigl\{A\in\GL_2(\cO_p)\mid\det(A)\in(\Z_p^\times)^{k-1}\bigr\}\subset\im(\rho_{f,p}). 
\]
In the proof of the following result, for which we refer to \cite{ribet2}, the requirement that $f$ be not CM comes into play.

\begin{theorem}[Ribet] \label{big-thm}
The representation $\rho_{f,p}$ has big image for all but finitely many $p$.
\end{theorem}

\begin{proof} Since $f$ is not CM, the theorem follows from \cite[Theorem 3.1]{ribet2}. \end{proof}

\subsubsection{Residual irreducibility} \label{residual-irreducibility-subsubsec}

With notation as above, if $\p$ is a prime of $\Q_f$ above $p$, then we denote by
\[ \rho_{f,\p}:G_\Q\longrightarrow\GL_2(\cO_\p) \]
the Galois representation attached to $f$ over $\cO_\p$. Reducing modulo the maximal ideal of $\cO_{\p}$, we obtain a residual representation 
\[ \bar\rho_{f,\p}:G_\Q\longrightarrow\GL_2(\F_\p), \] 
where $\F_\p\defeq\cO_\p/\p\cO_\p$ is the residue field of $\Q_f$ at $\p$. Moreover, write $\bar\rho_{f,\p}^\dagger$ for the self-dual twist of $\bar\rho_{f,\p}$, \emph{i.e.}, the $k/2$-fold Tate twist of $\bar\rho_{f,\p}$.

\begin{proposition} \label{residual-irreducibility-prop}
If $p$ is a prime number such that $\rho_{f,p}$ has big image, then the image of $\bar\rho_{f,\p}$ contains $\SL_2(\F_p)$ for every prime $\p$ of $\Q_f$ above $p$. In particular, $\bar\rho_{f,\p}$ and $\bar\rho_{f,\p}^\dagger$ are irreducible for every $\p$ above $p$.
\end{proposition}

\begin{proof} Let $p$ be a prime number such that $\rho_{f,p}$ has big image and let $\p$ be a prime of $\Q_f$ above $p$. Then the image of $\rho_{f,\p}$ contains $\SL_2(\Z_p)$, which implies that the image of $\bar\rho_{f,\p}$ contains $\SL_2(\F_p)$ and $\bar\rho_{f,\p}$ is irreducible. Finally, the irreducibility of a representation is preserved by tensorization with $1$-dimensional representations (see, \emph{e.g.}, \cite[Exercise 2.2.14, (2)]{kowalski}), so the irreducibility of $\bar\rho_{f,\p}^\dagger$ follows from that of $\bar\rho_{f,\p}$. \end{proof}

Combining Theorem \ref{big-thm} and Proposition \ref{residual-irreducibility-prop}, we get that, for all but finitely many prime numbers $p$, the representation $\bar\rho_{f,\p}^\dagger$ is irreducible for every prime $\p$ of $\Q_f$ above $p$ (\emph{cf.} also \cite[Theorem 2.1, (a)]{ribet2}).

\subsubsection{Trivial Nebentypus specializations}

Let $\omega:\F_p^\times\rightarrow\Bmu_{p-1}$ be the Teichm\"uller character. In the rest of this article, we will be interested in the even weight specializations $\hnf_\kappa$ of $\hf$ having trivial Nebentypus. Since the Nebentypus of $\hnf_\kappa$ is $\omega^{k-\kappa}$, this condition is tantamount to the congruence
\begin{equation} \label{cong-eq}
\kappa\equiv k\pmod{p-1}.
\end{equation}
Henceforth, we shall assume congruence \eqref{cong-eq}.

\begin{remark} \label{higher-rem}
If $\kappa>2$, then $\hnf_\kappa$ is the $p$-stabilization of a newform $f_\kappa^{(p),\circ}\in S_\kappa(\Gamma_0(N))$. As is customary in Hida theory, we shall not distinguish between $\hnf_\kappa$ and $f_\kappa^{(p),\circ}$; rather, we view $\hnf_\kappa$ as a newform of level $N$, as this is not restrictive for our goals (\emph{cf.} \cite[Lemma 2.10]{vigni-hida}).
\end{remark}

\subsection{The subset of prime numbers $\Sigma_f$} \label{sigma-subsec}

Given a prime number $p$, let us write $\bar\rho_{f,\fP}$ for the representation modulo $\fP\cap\cO$ attached to $f$, where $\fP$ is the prime of $\overline\Q$ lying over $p$ that was fixed in \S \ref{nc-subsec}. An analogous convention will be adopted for specializations of $p$-adic Hida families.

\subsubsection{The set $\Sigma_f$} \label{sigma-subsubsec}

The set $\Sigma_f$ that will appear in the statement of our main result consists of all the prime numbers $p$ such that
\begin{itemize}
\item $p\nmid2N$;
\item $p$ is ordinary for $f$;
\item $\bar\rho_{f,\fP}$ is irreducible.
\end{itemize}
By \S \ref{residual-irreducibility-subsubsec}, all but finitely many ordinary primes for $f$ belong to $\Sigma_f$. 

\subsubsection{$p$-distinguishedness}

Let $p\in\Sigma_f$. Let $\bar\omega_p$ be the mod $p$ cyclotomic character. By the theory of ordinary Galois representations, the restriction of $\bar\rho_{f,\fP}$ to the decomposition group $D_p$ at $p$ is reducible and can be written (up to equivalence) in an upper-triangular form as
\[
{\bar{\rho}_{f,\fP}|}_{D_p}\sim\begin{pmatrix}\delta_1 & * \\ 0 & \delta_2 \end{pmatrix},
\]
where $\delta_1,\delta_2:D_p\rightarrow\overline{\F_p}^\times$ are continuous characters such that $\delta_2$ is unramified and, since the Nebentypus of $f$ is trivial, $\delta_1\delta_2=\bar\omega_p^{k-1}$ (see, \emph{e.g.}, \cite[Theorem 2.1.4]{Wiles-Ordinary}). Therefore, writing $I_p\subset D_p$ for the inertia subgroup, ${\delta_1|}_{I_p}=\bar\omega_p^{k-1}$, which implies, since $k\not\equiv1\pmod{p-1}$, that ${\delta_1|}_{I_p}$ is not trivial. This shows that $\delta_1\not=\delta_2$, \emph{i.e.}, $\bar\rho_{f,\fP}$ is \emph{$p$-distinguished}. 

\subsubsection{Consequences on big Galois representations} \label{consequences-subsubsec}

Let $p\in\Sigma_f$. With notation as above, the fact that $\bar\rho_{f,\fP}$ is irreducible and $p$-distinguished has two important consequences:
\begin{itemize}
\item $\T_{\hf}$ is free of rank $2$ over $\mathcal R$ (\cite[Th\'eor\`eme 7]{mazur-tilouine});
\item $\overline{\T}_{\hf}\sim\bar\rho_{\hnf_\wp,\fP}$ after a finite base change for every arithmetic prime $\wp$ of $\mathcal R$ (see, \emph{e.g.}, \cite[Proposition 5.4]{LV-MM}).
\end{itemize}
From now on, we shall freely use these results with no further comment.

\subsection{The imaginary quadratic field $K$} \label{algebraic-subsec}

We introduce the imaginary quadratic fields that will appear in S.-W. Zhang's higher weight Gross--Zagier formula, which will play a crucial role in our arguments.

\subsubsection{Choice of $K$} \label{K-subsubsec}

Fix an imaginary quadratic field $K$ such that every prime factor of $Np$ splits in $K$ (in other words, $K$ satisfies the Heegner hypothesis relative to $Np$): the existence of (infinitely many) imaginary quadratic fields satisfying this condition is a classical result in algebraic number theory. The discriminant (respectively, class number) of $K$ will be denoted by $D_K$ (respectively, $h_K$); notice that $D_K$ is coprime to $N$. We write $\cO_K$ for the ring of integers of $K$ and $\eta_K$ for the (quadratic) Dirichlet character attached to $K$. 

\subsubsection{Analytic ranks and root numbers over $K$}

As in \S \ref{analytic-subsubsec}, let $\cg\in S_h(\Gamma_0(N))$ be a newform of weight $h\geq2$. Let $\Lambda(\cg,s)$ and $\Lambda(\cg\otimes\eta_K,s)$ be the completed $L$-functions of $\cg$ and $\cg\otimes\eta_K$, respectively; for the purposes of this paper, we define the completed $L$-function of the base change $\cg/K$ to be 
\begin{equation} \label{Lambda-eq}
\Lambda(\cg/K,s)\defeq\Lambda(\cg,s)\cdot\Lambda(\cg\otimes\eta_K,s).
\end{equation}
It follows that $\Lambda(\cg/K,s)$ satisfies the functional equation
\[
\Lambda(\cg/K,s)=W(\cg/K)\cdot\Lambda(\cg/K,h-s),
\]
where, as a consequence of \eqref{Lambda-eq}, the (global) root number of $\cg/K$ is the product
\begin{equation} \label{W-eq1}
W(\cg/K)=W(\cg)\cdot W(\cg\otimes\eta_K)\in\{\pm1\}
\end{equation}
of the root numbers of $\cg$ and $\cg\otimes\eta_K$; observe that $\cg\otimes\eta_K$ has trivial Nebentypus too, so its root number is either $+1$ or $-1$.

\begin{definition}
The \emph{analytic rank} of $\cg/K$ is $r_\an(\cg/K)\defeq\ord_{s=h/2}L(\cg/K,s)\in\N$.
\end{definition}

As in \S \ref{analytic-subsubsec}, there is an equality $W(\cg/K)=(-1)^{r_\an(\cg/K)}$.

\begin{proposition} \label{W-prop}
$W(\cg\otimes\eta_K)=-W(\cg)$.
\end{proposition}

\begin{proof} It is a classical result of Weil (see, \emph{e.g.}, \cite[Theorem 6]{li}) that there is an equality 
\begin{equation} \label{W-eq2}
W(\cg\otimes\eta_K)=\eta_K(-N)\cdot W(\cg).
\end{equation}
On the other hand, $\eta_K(-1)=-1$ because the quadratic field $K$ is imaginary, whereas $\eta_K(N)=1$ since $K$ satisfies the Heegner hypothesis relative to $Np$. It follows from \eqref{W-eq2} that $W(\cg\otimes\eta_K)=-W(\cg)$, as desired. \end{proof}

\begin{corollary} \label{W-coro}
$W(\cg/K)=-1$.
\end{corollary}

\begin{proof} Combine Proposition \ref{W-prop} and equality \eqref{W-eq1}. \end{proof}

\begin{corollary} \label{W-coro2}
The integer $r_\an(\cg/K)$ is odd. In particular, $L(\cg/K,h/2)=0$.
\end{corollary}

\begin{proof} Immediate from Corollary \ref{W-coro} and the equality $W(\cg/K)=(-1)^{r_\an(\cg/K)}$. \end{proof}

\subsection{Zhang's formula of Gross--Zagier type} \label{zhang-subsec}

We recall a Gross--Zagier type formula due to S.-W. Zhang for even weight cusp forms with trivial Nebentypus (\cite{Zhang-heights}): as will be apparent, this will be a key ingredient in our arguments. 


\subsubsection{Heegner cycles} \label{heegner-subsubsec}

Let $\cg\in S_h(\Gamma_0(N))$ be a newform of weight $h\geq4$, level $N$ and trivial Nebentypus. In particular, what we say below applies to the newforms $f_\kappa\in S_\kappa(\Gamma_0(N))$ for all $\kappa\geq4$. Let $X_{h-2}$ denote the Kuga--Sato variety of level $N$ and weight $h$; let $K$ be the imaginary quadratic field from \S \ref{K-subsubsec}. Moreover, denote by $\T$ the Hecke algebra acting on $S_h(\Gamma_0(N))$; by applying the idempotent $e_{\cg}\in\T\otimes\Q_{\cg}$ corresponding to $\cg$, for any number field $L$ we can consider the $\cg$-isotypic component
\[
\cCH_L(\cg)\defeq e_{\cg}\Bigl(\CH^{h/2}(X_{h-2}/L{)}_0\otimes_\Q\Q_{\cg}\Bigr),
\]
which is a $\Q_{\cg}$-vector subspace of $\CH^{h/2}(X_{h-2}/L{)}_0\otimes_\Q\Q_{\cg}$.

Now let $z_{\cg,K}\in\cCH_K(\cg)$ be the Heegner cycle over $K$ attached to $\cg$ that was originally defined by Nekov\'a\v{r} (\cite{Nek}; \emph{cf.} also \cite{Nek2}); namely, if $H_K$ is the Hilbert class field of $K$ and $z_{\cg,1}\in\cCH_{H_K}(\cg)$ is the Heegner cycle of conductor $1$ attached to $\cg$, then 
\begin{equation} \label{z-K-eq}
z_{\cg,K}\defeq\cores_{H_K/K}(z_{\cg,1}),
\end{equation}
where $\cores_{H_K/K}$ stands for corestriction.

\subsubsection{S.-W. Zhang's formula of Gross--Zagier type} \label{zhang-subsubsec}

Using arithmetic intersection theory \emph{\`a la} Gillet--Soul\'e (\cite{GS-1}, \cite[Ch. III]{soule}), S.-W. Zhang defined in \cite{Zhang-heights} a pairing, which in this paper will be denoted by ${\langle\cdot,\cdot\rangle}_{\cg,\GS}$, between certain $\cg$-isotypic Heegner-type cycles on Kuga--Sato varieties (\emph{cf.} also \cite[\S 5.1.3]{LV-RMS}). In terms of it, Zhang proved an analogue for the $L$-function of $\cg$ of the Gross--Zagier formula for weight-$2$ newforms (\cite{GZ}). Let $(\cg,\cg)$ be the Petersson inner product of $\cg$ with itself and set $u_K\defeq\#\cO_K^\times/2$. Zhang's formula expresses the critical value of the derivative of $L(\cg/K,s)$ in terms of a suitable combination $s_\cg'$ of Heegner cycles in $\R$-linear Chow groups of Kuga--Sato varieties fibered over (classical) modular curves.


\begin{theorem}[S.-W. Zhang] \label{zhang-thm}
$L'(\cg/K,h/2)=\begin{displaystyle}\frac{2^{2h-1}\pi^k(\cg,\cg)}{(h-2)!u_K^2h_K\sqrt{|D_K|}}\end{displaystyle}\cdot\big\langle s'_\cg,s'_\cg\big\rangle_{\cg,\GS}$.
\end{theorem}

\begin{proof} With notation as in \cite{Zhang-heights}, this formula follows from \cite[Corollary 0.3.2]{Zhang-heights} upon taking $\chi$ to be the trivial character. \end{proof}

When $\cg=f_\kappa$ for an even integer $\kappa\geq4$, we shall write ${\langle \cdot,\cdot\rangle}_{\kappa,\GS}$ instead of ${\langle\cdot,\cdot\rangle}_{f_\kappa,\GS}$.

\subsubsection{Refined choice of $K$} \label{refined-subsubsec}

From here on, choose an imaginary quadratic field $K$ such that
\begin{enumerate}
\item every prime factor of $Np$ splits in $K$;
\item $r_\an(f\otimes\eta_K)=0$.
\end{enumerate}
By assumption, $r_\an(f)=1$, whence $W(f)=-1$; the existence of an imaginary quadratic field $K$ satisfying (1) and (2) is then ensured by \cite[p. 543, Theorem, (ii)]{BFH} (\emph{cf.} also \cite{Waldspurger}). Since $r_\an(f)=1$ by assumption, the splitting of $L$-functions 
\[
L(f/K,s)=L(f,s)\cdot L(f\otimes\eta_K,s)
\]
gives $L(f/K,k/2)=0$ and an equality
\[ 
L'(f/K,k/2)=L'(f,k/2)\cdot L(f\otimes\eta_K,k/2).
\]
Moreover, $r_\an(f\otimes\eta_K)=0$ by our choice of $K$, and then $r_{\an}(f/K)=1$.

As in \S \ref{heegner-subsubsec}, let $z_{f,K}$ be the Heegner cycle over $K$ attached to $f$. Recall that we have fixed an embedding $\Q_f\hookrightarrow\R$.

\begin{proposition} \label{z-prop}
The cycle $z_{f,K}$ is not trivial.
\end{proposition}

\begin{proof} As we pointed out above, our choice of $K$ gives $r_\an(f/K)=1$, and then Theorem \ref{zhang-thm} with $\cg=f$ shows that $s'_f\not=0$. On the other hand, as a consequence of a result of Scholl (\cite{scholl}), the $f$-isotypic component of the Chow motive of $X_{k-2}$ is simple over $\Q_f$, so Schur's lemma implies that the projectors used by Nekov\'a\v{r} and by Zhang in their constructions of the cycles $z_{f,K}$ and $s_f'$, respectively, act on this subspace as nonzero scalars. Consequently, $s_f'=\alpha\cdot z_{f,K}$ for some $\alpha\in\Q_f^\times$, whence $z_{f,K}\not=0$. \end{proof}

\subsection{Big Heegner points and interpolation} \label{big-subsec}

We briefly overview the interpolation of (classical) Heegner cycles, due to Castella (\cite{CasHeeg}) and Ota (\cite{ota-JNT}), via Howard's big Heegner points (\cite{Howard-Inv}).
 
\subsubsection{$p$-adic Abel--Jacobi maps} \label{AJ-subsubsec}

Recall that we have fixed an embedding $\overline\Q\hookrightarrow\overline{\Q_p}$, which determines a prime ideal $\fP$ of the ring of integers of $\overline\Q$ lying over $p$. If $L$ is a number field inside $\overline\Q$, then $\fP$ induces in turn a prime $\p_L$ of $L$ over $p$; to ease our notation, we write $L_\fP$ for the completion of $L$ at $\p_L$.

As above, let $\hf$ be the $p$-adic Hida family passing through $f$ and let $\kappa\geq2$ be an even integer; since $p$ is understood, we will usually write $\f$ in place of $\hf$. With notation from \S \ref{heegner-subsubsec} in force, let $V^\dagger_{f_\kappa}\defeq V_{f_\kappa}(\kappa/2)$ be the self-dual twist of $V_{f_\kappa}$ and, for a given number field $L$, let 
\begin{equation} \label{AJ-eq}
\AJ_{p,L,\kappa}:\cCH_L(f_\kappa)\otimes_{\Q_{f_\kappa}}\!\Q_{f_\kappa,\fP}\longrightarrow H^1_f\bigl(L,V^\dagger_{f_\kappa}\bigr)\subset H^1\bigl(L,V^\dagger_{f_\kappa}\bigr)
\end{equation}
be the $p$-adic (or, rather, $\fP$-adic) Abel--Jacobi map over $L$ attached to $f_\kappa$, which is $\Q_{f_\kappa,\fP}$-linear; here $H^1_f(\star,\bullet)$ denotes a Selmer group \emph{\`a la} Bloch--Kato (\cite{BK}).

Let $\langle z_{f,K}\rangle$ be the $\Q_f$-vector subspace of $\cCH_K(f)$ generated by $z_{f,K}$. Specializing \eqref{AJ-eq} to $\kappa=k$ (and then $f_\kappa=f$, \emph{cf.} Remark \ref{higher-rem}), our main result will be proved under 

\begin{assumption} \label{AJ-ass}
The map $\AJ_{p,K,k}$ is injective on $\langle z_{f,K}\rangle$.
\end{assumption}

This is a very special instance of a general conjecture in arithmetic algebraic geometry predicting the injectivity of $p$-adic Abel--Jacobi maps for proper, smooth algebraic varieties over number fields (see, \emph{e.g.}, \cite[Conjecture 2.1, (2)]{Nek-CRM}).

\subsubsection{Big Heegner points and interpolation} \label{int-subsubsec}

For every even integer $\kappa\geq4$ and number field $L$, recall the specialization map $\rho_{L,\kappa}$ from \S \ref{spec-cohom-subsubsec}. Let $\mathfrak{Z}_{\f,H_K}\in H^1\bigl(H_K,\T_{\f}^\dagger\bigr)$ be the big Heegner point of conductor $1$ introduced by Howard in \cite{Howard-Inv}. Denote by $\overline{\Z_p}$ the integral closure of $\Z_p$ in $\overline{\Q_p}$. 

The theorem below describes an interpolation property enjoyed by $\mathfrak{Z}_{\f,H_K}$; actually, this is a special case of a more general specialization result for big Heegner points of arbitrary conductor.

\begin{theorem}[Castella, Ota] \label{int-thm}
Let $\kappa\geq4$ be an even integer such that $\kappa\equiv k\pmod{p-1}$. There exists $d(\kappa)\in\overline{\Z_p}^\times$ such that the equality
\begin{equation} \label{ota-eq}
\rho_{H_K,\kappa}(\mathfrak Z_{\f,H_K})=d(\kappa)\cdot\AJ_{p,H_K,\kappa}(z_{f_\kappa,1})
\end{equation}
holds in $H^1\bigl(H_K,V^\dagger_{f_\kappa}\bigr)$.
\end{theorem}

Clearly, in \eqref{ota-eq} we are using the same symbol to denote $z_{f_\kappa,1}\in\cCH_{H_K}(f_\kappa)$ and its image in $\cCH_{H_K}(f_\kappa)\otimes_{\Q_{f_\kappa}}\!\Q_{f_\kappa,\fP}$.

\begin{proof} This is \cite[Theorem 1.2]{ota-JNT} (see also \cite[Theorem 6.5]{CasHeeg}). \end{proof}

\begin{remark}
A specialization result in weight $2$ analogous to Theorem \ref{int-thm} is also available: see, \emph{e.g.}, \cite[Remark 6.6]{CasHeeg}.
\end{remark}

Now set
\begin{equation} \label{big-K-eq}
\mathfrak Z_{\f,K}\defeq\cores_{H_K/K}(\mathfrak Z_{\f,H_K})\in H^1\bigl(K,\T_{\f}^\dagger\bigr).
\end{equation}
For an even integer $\kappa\geq4$, recall the Heegner cycle $z_{f_\kappa,K}$ over $K$ defined (with $\cg=f_\kappa$) in \eqref{z-K-eq}. The following result is a formal consequence of Theorem \ref{int-thm}.

\begin{corollary} \label{int-coro}
Let $\kappa\geq4$ be an even integer such that $\kappa\equiv k\pmod{p-1}$. There exists $d(\kappa)\in\overline{\Z_p}^\times$ such that the equality
\[
\rho_{K,\kappa}(\mathfrak Z_{\f,K})=d(\kappa)\cdot\AJ_{p,K,\kappa}(z_{f_\kappa,K})
\]
holds in $H^1\bigl(K,V^\dagger_{f_\kappa}\bigr)$.
\end{corollary}

\begin{proof} Standard functoriality properties of $p$-adic Abel--Jacobi maps and of specialization maps ensure that these maps commute with cohomological corestriction, so the corollary follows by combining Theorem \ref{int-thm} with \eqref{z-K-eq} and \eqref{big-K-eq}. \end{proof}

Our next goal is to prove that $\mathfrak Z_{\f,K}$ is not torsion over $\mathcal R$; for this purpose, we need two auxiliary algebraic results.

\begin{lemma} \label{2-lemma}
Let $F$ be a field and let $G$ be a group acting irreducibly on a $2$-dimensional $F$-vector space $V$. If $H$ is a subgroup of $G$ of index $2$, then $V^H=\{0\}$.
\end{lemma}

\begin{proof} By contradiction, suppose $V^H\not=\{0\}$. Pick $v\in V^H\smallsetminus\{0\}$, fix $c\in G\smallsetminus H$ and consider the $F$-vector subspace $W\defeq\big\langle v,c(v)\big\rangle$ of $V$ spanned by $v$ and $c(v)$. Because $c^2\in H$ and $H$, which has index $2$, is normal in $G$, it is straightforward to check that $W$ is $G$-stable. Since $W\not=\{0\}$ and the action of $G$ on $V$ is irreducible, it follows that $W=V$; in other words, $\{v,c(v)\}$ is a basis of $V$ over $F$. The normality of $H$ forces $H$ to act trivially on $V$, which shows that the action of $G$ on $V$ induces an action of $G/H$ on $V$. On the other hand, this action of $G/H$ will be irreducible, which is impossible because $G/H\simeq\Z/2\Z$ is abelian and $\dim_F(V)>1$. \end{proof}

\begin{proposition} \label{torsion-free-prop}
The $\mathcal R$-module $H^1\bigl(K,\T_{\f}^\dagger\bigr)$ is torsion-free.
\end{proposition}

\begin{proof} We divide the proof into three steps. To lighten our notation, set $\m\defeq\m_{\mathcal R}$.

\texttt{Step 1}. By assumption, $G_\Q$ acts irreducibly on $\T_{\f}/\m\T_{\f}$. Moreover, one can check that $\T_{\f}^\dagger\big/\m\T_{\f}^\dagger$ is (equivalent to) the critical twist of $\T_{\f}/\m\T_{\f}$, and then $\T_{\f}^\dagger\big/\m\T_{\f}^\dagger$ is irreducible as a representation of $G_\Q$. Therefore, by Lemma \ref{2-lemma}, $H^0\bigl(K,\T_{\f}^\dagger\big/\m\T_{\f}^\dagger\bigr)=\{0\}$. From here until the end of this proof, set $\T\defeq\T_{\f}^\dagger$.

\texttt{Step 2.} We want to show that $H^0(K,\T)=\{0\}$. Arguing by contradiction, suppose that there is $x\in H^0(K,\T)\smallsetminus\{0\}$. Since $\T$ is finitely generated over the noetherian local ring $(\mathcal R,\m)$, Krull's intersection theorem guarantees that
\[
\bigcap_{i\geq0}\m^i\T=\{0\}.
\]
Set $j\defeq\max\bigl\{i\in\N\mid x\in\m^i\T\bigr\}\in\N$; if $\bar x$ is the image of $x$ in $\m^j\T\big/\m^{j+1}\T$, then $\bar x\not=0$. On the other hand, the freeness of $\T$ over $\mathcal R$ yields an isomorphism of $\mathcal R/\m$-vector spaces
\[
\m^j\T/\m^{j+1}\T\simeq\bigl(\m^j/\m^{j+1}\bigr)\otimes_{\mathcal R/\m}(\T/\m\T),
\]
which is $G_K$-equivariant. Since the action of $G_K$ on $\T$ is $\mathcal R$-linear, $G_K$ acts trivially on $\m^j/\m^{j+1}$, which shows that the $G_K$-module $\m^j\T/\m^{j+1}\T$ is isomorphic to a direct sum of finitely many copies of $\T/\m\T$. By \texttt{Step 1}, $H^0(K,\T/\m\T)$ is trivial, so $H^0\bigl(K,\m^j\T/\m^{j+1}\T\bigr)$ is trivial as well. However, $\bar x\in\m^j\T\big/\m^{j+1}\T$ must be $G_K$-invariant because $x$ is, which contradicts the fact that $\bar x\not=0$. We conclude that $H^0(K,\T)=\{0\}$.

\texttt{Step 3.} Now let $r\in\mathcal R\smallsetminus\{0\}$: we want to prove that the multiplication-by-$r$ map on $H^1(K,\T)$ is injective. Consider the short exact sequence 
\begin{equation} \label{T-short-eq}
0\longrightarrow\T\overset{r\cdot}\longrightarrow\T\longrightarrow\T/r\T\longrightarrow0
\end{equation}
of left $\mathcal R[G_K]$-modules. Passing to cohomology in \eqref{T-short-eq}, the vanishing result of \texttt{Step 2} yields an exact sequence
\[
0\longrightarrow H^0(K,\T/r\T)\longrightarrow H^1(K,\T)\overset{r\cdot}\longrightarrow H^1(K,\T)\longrightarrow\dots
\]
Thus, we need to show that $H^0(K,\T/r\T)$ is trivial. Of course, we can assume $r\notin\m$. Let $M\defeq\T/r\T$. Since $\T$ is a finitely generated $\mathcal R$-module, $M$ is finitely generated over $\mathcal R/r\mathcal R$. Suppose, by contradiction, that there exists $x\in H^0(K,M)\smallsetminus\{0\}$. By Krull's intersection theorem, we have
\[
\bigcap_{i\geq0}\m^iM=\{0\}.
\]
Since $x\neq0$, there is a unique integer $j\in\N$ such that $x\in\m^jM$ but $x\notin\m^{j+1}M$; thus, the image $\bar x$ of $x$ in $\m^jM/\m^{j+1}M$ is not trivial and $G_K$-invariant. Observe that there is a $G_K$-equivariant surjection
\[
\m^j\T/\m^{j+1}\T\longepi\m^jM/\m^{j+1}M.
\]
We showed in \texttt{Step 2} that $\m^j\T/\m^{j+1}\T$ is isomorphic as a representation of $G_K$ to a direct sum of finitely many copies of $V\defeq\T/\m\T$, so $\m^jM/\m^{j+1}M$ is a $G_K$-equivariant image of $V^{\oplus n}$ for some integer $n\geq1$. On the other hand, $V$ is an irreducible $G_K$-representation, so any quotient of $V^{\oplus n}$ is completely reducible and isomorphic to $V^{\oplus m}$ for some $m\in\{1,\dots,n\}$. Therefore, there is a $G_K$-equivariant isomorphism
\[
\m^jM\big/\m^{j+1}M\simeq V^{\oplus m}
\]
for some integer $m\geq1$. Since, by \texttt{Step 1}, $H^0(K,V)$ is trivial, $H^0\bigl(K,\m^j\T/\m^{j+1}\T\bigr)$ is also trivial, which contradicts the fact that $\bar x\not=0$. We conclude that $H^0(K,M)=\{0\}$. \end{proof}

Recall that Assumption \ref{AJ-ass} is in force.

\begin{proposition} \label{non-torsion-prop}
The big Heegner cycle $\mathfrak Z_{\f,K}$ is not $\mathcal R$-torsion.
\end{proposition}

\begin{proof} By Proposition \ref{z-prop}, the Heegner cycle $z_{f,K}$ is not trivial, and then Assumption \ref{AJ-ass} ensures that $\AJ_{p,K,k}(z_{f,K})\not=0$. Therefore, $\mathfrak Z_{\f,K}\not=0$ by Corollary \ref{int-coro}. On the other hand, $H^1\bigl(K,\T_{\f}^\dagger\bigr)$ is torsion-free over $\mathcal R$ by Proposition \ref{torsion-free-prop}, so $\mathfrak Z_{\f,K}$ is not $\mathcal R$-torsion. \end{proof}

\begin{remark}
Proposition \ref{non-torsion-prop} confirms \cite[Conjecture 3.4.1]{Howard-Inv} in our analytic rank $1$ setting and under Assumption \ref{AJ-ass}.
\end{remark}

\subsection{Main result} \label{main-subsec}

We can finally prove our result on analytic rank $1$ propagation in Hida families of modular forms. In addition to Assumption \ref{AJ-ass}, we work under

\begin{assumption} \label{GS-ass}
The pairing ${\langle\cdot,\cdot\rangle}_{\kappa,\GS}$ is positive definite for every even integer $\kappa\geq4$ satisfying $\kappa\equiv k\pmod{p-1}$.
\end{assumption}

\begin{remark}
The positive definiteness of ${\langle\cdot,\cdot\rangle}_{\kappa,\GS}$ is a consequence of one of the arithmetic analogues of the standard conjectures proposed by Gillet and Soul\'e (\cite[Conjecture 2]{GS-2}); it is also a special case of general conjectures of Beilinson (\cite{beilinson-height}) and Bloch (\cite{bloch-height}) on positive definiteness of height pairings. While it is natural to impose such a positive definiteness condition when studying the arithmetic of Heegner cycles (see, \emph{e.g.}, \cite[Assumption 4.1]{Xue}), we are not aware of any result in this direction in our higher weight setting.
\end{remark}

Recall the set of prime numbers $\Sigma_f$ from \S \ref{sigma-subsubsec}, which consists of all but finitely many ordinary primes for $f$. Now we can state and prove our main result.

\begin{theorem} \label{main-thm}
Let $p\in\Sigma_f$. Under Assumptions \ref{AJ-ass} and \ref{GS-ass}, $r_\an(f_\kappa)=1$ for all but finitely many even integers $\kappa\geq2$ such that $f_\kappa$ has trivial Nebentypus.
\end{theorem}

It is worth pointing out that the even integers $\kappa\geq4$ such that $f_\kappa$ has trivial Nebentypus are exactly those with $\kappa\equiv k\pmod{p-1}$. Recall from \S \ref{hida-subsec} that $\mathtt{ArithSpec}(\mathcal R)$ is the set of all the arithmetic primes of $\mathcal R$.

\begin{proof} Let $p\in\Sigma_f$. As above, let us set $\f\defeq\hf$ and $f_\kappa\defeq\hnf_\kappa$. Let $\mathscr M\defeq\widetilde{H}^1_f\bigl(K,\T_{\f}^\dagger\bigr)$ be Nekov\'a\v{r}'s extended Selmer group of $\T_{\f}^\dagger$ over $K$; by \cite[Proposition 4.2.3]{Nek-Selmer} and \cite[Theorem 8.3.20]{NSW}, the $\mathcal R$-module $\mathscr M$ is finitely generated. Moreover, since $N$ and the discriminant of $K$ are coprime, \cite[Proposition 2.4.5]{Howard-Inv} ensures that the big Heegner cycle $\mathfrak Z_{\f,K}$ lives in the strict Greenberg Selmer group $\Sel_{\mathrm{Gr}}\bigl(K,\T_{\f}^\dagger\bigr)$ defined in \cite[Definition 2.4.2]{Howard-Inv}. However, there is a canonical isomorphism $\mathscr M\simeq\Sel_{\mathrm{Gr}}\bigl(K,\T_{\f}^\dagger\bigr)$ (see, \emph{e.g.}, \cite[eq. (21)]{Howard-Inv}), so we can view $\mathfrak Z_{\f,K}\in\mathscr M$. By Proposition \ref{non-torsion-prop}, $\mathfrak Z_{\f,K}$ is not $\mathcal R$-torsion, so \cite[Lemma 2.1.7]{Howard-Inv} allows us to conclude that $\mathfrak{Z}_{\f,K}\notin\wp\mathscr M_\wp$ for all but finitely many $\wp\in\mathtt{ArithSpec}(\mathcal R)$. Let us consider the finite set
\[
\Xi_{\f}\defeq\bigl\{\wp\in\mathtt{ArithSpec}(\mathcal R)\mid\mathfrak Z_{\f,K}\in\wp\mathscr M_\wp\bigr\}.
\]
Let $\kappa\geq4$ be an even integer whose corresponding arithmetic prime $\wp_\kappa$ does not belong to $\Xi_{\f}$. Since $\min\{k,\kappa\}>2$, the root numbers of $f$ and $f_\kappa$ are equal (see, \emph{e.g.}, \cite[Lemma 3.5]{vigni-hida}), \emph{i.e.}, $W(f_\kappa)=W(f)=-1$; it follows that $r_\an(f_\kappa)$ is odd, hence
\begin{equation} \label{r-an-eq}
r_\an(f_\kappa)\geq1.
\end{equation}
Because localization at $\wp_\kappa$ is an exact functor, the formation of extended Selmer groups commutes with it, so there is a canonical isomorphism $\mathscr M_{\wp_\kappa}\simeq\widetilde{H}^1_f\bigl(K,\T_{\f,\wp_\kappa}^\dagger\bigr)$ that we shall regard as an identification. Thus, $\mathfrak Z_{\f,K}\notin\wp_\kappa\widetilde{H}^1_f\bigl(K,\T_{\f,\wp_\kappa}^\dagger\bigr)$.

By \cite[Lemma 2.1.6]{Howard-Inv}, $\mathcal R_{\wp_\kappa}$ is a discrete valuation ring: fix a uniformizer $\varpi_\kappa\in\wp_\kappa\mathcal R_{\wp_\kappa}$ of $\mathcal R_{\wp_\kappa}$. As pointed out in the proof of \cite[Corollary 3.4.3]{Howard-Inv}, the short exact sequence of $G_\Q$-modules
\[
0\longrightarrow\T_{\f,\wp_\kappa}^\dagger\xlongrightarrow{\varpi_\kappa\cdot}\T_{\f,\wp_\kappa}^\dagger\longrightarrow V_{f_\kappa}^\dagger\longrightarrow0
\]
induces an exact sequence of extended Selmer groups, which yields an injection
\begin{equation} \label{final-eq1}
0\longrightarrow\widetilde{H}^1_f\bigl(K,\T_{\f,\wp_\kappa}^\dagger\bigr)\big/\wp_\kappa \widetilde{H}^1_f\bigl(K,\T_{\f,\wp_\kappa}^\dagger\bigr)\longrightarrow\widetilde{H}^1_f\bigl(K,V_{f_\kappa}^\dagger\bigr).
\end{equation}
The isomorphism $\mathscr M\simeq\Sel_{\mathrm{Gr}}\bigl(K,\T_{\f}^\dagger\bigr)$ recalled above gives a canonical isomorphism 
\begin{equation} \label{final-eq2}
\widetilde{H}^1_f\bigl(K,\T_{\f,\wp_\kappa}^\dagger\bigr)\simeq\Sel_{\mathrm{Gr}}\bigl(K,\T_{\f,\wp_\kappa}^\dagger\bigr)\subset H^1\bigl(K,\T_{\f,\wp_\kappa}^\dagger\bigr).
\end{equation}
On the other hand, there are canonical isomorphisms
\begin{equation} \label{final-eq3}
\widetilde{H}^1_f\bigl(K, V_{f_\kappa}^\dagger\bigr)\simeq\Sel_{\mathrm{Gr}}\bigl(K, V_{f_\kappa}^\dagger\bigr)=H^1_f\bigl(K,V_{f_\kappa}^\dagger\bigr)\subset H^1\bigl(K,V_{f_\kappa}^\dagger\bigr)
\end{equation}
(see, \emph{e.g.}, \cite[eqs. (22) and (23)]{Howard-Inv}). Combining \eqref{final-eq2} and \eqref{final-eq3}, we see that injection \eqref{final-eq1} is induced by the restriction to $\widetilde{H}^1_f\bigl(K,\T_{\f,\wp_\kappa}^\dagger\bigr)$ of the specialization map  $\rho_{K,\kappa}$ from \eqref{spec-eq2}. Since $\mathfrak Z_{\f,K}\notin\wp_\kappa\widetilde{H}^1_f\bigl(K,\T_{\f,\wp_\kappa}^\dagger\bigr)$, we conclude that $\rho_{K,\kappa}(\mathfrak Z_{\f,K})\not=0$, and then $z_{f_\kappa,K}\neq0$ by Corollary \ref{int-coro}. With notation from \S \ref{zhang-subsubsec} in force, the arguments in the proof of Proposition \ref{z-prop} show that $s'_{f_\kappa}\neq0$, so a combination of Theorem \ref{zhang-thm} and Assumption \ref{GS-ass} ensures that $L'(f_\kappa/K,\kappa/2)\not=0$. Bearing in mind that $L(f_\kappa,\kappa/2)=0$, the splitting of $L$-functions $L(f_\kappa/K,s)=L(f_\kappa,s)\cdot L(f_\kappa\otimes\eta_K,s)$ yields the equality
\[
L'(f_\kappa/K,\kappa/2)=L'(f_\kappa,\kappa/2)\cdot L(f_\kappa\otimes\eta_K,\kappa/2),
\]
whence 
\begin{equation} \label{final-eq4}
L'(f_\kappa,\kappa/2)\not=0.
\end{equation}
Finally, combining \eqref{r-an-eq} and \eqref{final-eq4} gives $r_\an(f_\kappa)=1$, as was to be shown. \end{proof}

As was remarked in the introduction, Theorem \ref{main-thm} provides evidence (in rank $1$) for Greenberg's ``minimality conjecture'' for analytic ranks in families of modular forms (\cite[p. 101]{Greenberg-CRM}). See, \emph{e.g.}, \cite[\S 3.2]{vigni-hida} for more details on Greenberg's conjecture for Hida families.

\begin{remark} \label{0-rem}
Under similar assumptions, our strategy of proof of Theorem \ref{main-thm} can be adapted to yield an analogous result on the propagation of analytic rank $0$. More precisely, assuming $r_\an(f)=0$, one chooses an auxiliary imaginary quadratic field $K$ such that every prime factor of $Np$ splits in $K$ and $r_\an(f\otimes\eta_K)=1$, which can be done thanks to \cite[p.543, Theorem, (i)]{BFH} (\emph{cf.} also \cite{MM-derivatives}). Once $K$ has been fixed, one can proceed \emph{mutatis mutandis} as in the proof of Theorem \ref{main-thm}. However, such a rank $0$ statement can also be obtained (with no need for counterparts of Assumptions \ref{AJ-ass} and \ref{GS-ass}) by combining properties of the Mazur--Kitagawa two-variable $p$-adic $L$-function (\cite{kitagawa}) with work of Kato (\cite{Kato}), as explained in \cite[Theorem 7]{Howard-derivatives}. Therefore, we decided not to elaborate on this case here.
\end{remark}

\bibliographystyle{amsplain}
\bibliography{rank1}

\end{document}